\documentclass[12pt,letterpaper]{amsart}

\makeatletter
\providecommand{\@LN}[2]{}
\makeatother  

\usepackage{epsfig,amsmath,amsthm,amssymb,graphicx}
\usepackage[top=1in, bottom=1in, left=1in, right=1in]{geometry}
\usepackage{color}

\usepackage{thmtools}      
\usepackage{thm-restate}   

\usepackage[pagebackref,hypertexnames=false,colorlinks=true]{hyperref} 
\usepackage{cleveref}

\usepackage{xfrac}
\usepackage[table]{xcolor}

\usepackage{comment}
\usepackage{array, multirow}
\usepackage{makecell}
\usepackage{import}

\usepackage{enumerate}

\usepackage{caption}
\usepackage{subcaption}

\usepackage{faktor}

\usepackage{tikz}
\usepackage{tikz-cd}
\usetikzlibrary{matrix, arrows, cd}
\usepackage{pb-diagram}
\usepackage{cancel}

\definecolor{myOlive}{rgb}{.29,.28,.16}
\definecolor{myRed}{rgb}{.78,0,0}

\newtheorem{proposition}{Proposition}[section]
\newtheorem{theorem}[proposition]{Theorem}
\newtheorem{corollary}[proposition]{Corollary}

\newtheorem*{theorem*}{Theorem}
\newtheorem*{proposition*}{Proposition}
\newtheorem*{lemma*}{Lemma}
\newtheorem*{corollary*}{Corollary}

\theoremstyle{definition}

\newtheorem{question}[proposition]{Question}

\usepackage{thmtools}      
\usepackage{thm-restate}   

\theoremstyle{remark}

\newcommand{\bdry}{\partial}

\newcommand{\N}{\mathbb{N}}
\newcommand{\Z}{\mathbb{Z}}
\newcommand{\F}{\mathcal{F}}

\newcommand{\lk}{\operatorname{lk}}

\renewcommand{\int}[1]{\operatorname{int}(#1)}

\newcommand{\Sum}{\displaystyle \sum  }

\newcommand{\ceil}[1]{\left\lceil #1 \right\rceil}

\begin{document}

\title{Gordian distance and clasper surgery for links}
\thanks{This material is based upon work supported by the National Science Foundation under Grant No. DMS-1928930 while the authors participated in a program hosted by the Simons Laufer Mathematical Sciences Institute (formerly Mathematical Sciences Research Institute) in Berkeley, California, during the summer of 2025.}

\author[A.~Bosman]{Anthony Bosman}
\address{Department of Mathematics and Physics, Andrews University}
\email{bosman@andrews.edu}
\urladdr{andrews.edu/cas/math/faculty/bosman-anthony.html}

\author[C.~W.~Davis]{Christopher W.\ Davis}
\address{Department of Mathematics, University of Wisconsin--Eau Claire}
\email{daviscw@uwec.edu}
\urladdr{people.uwec.edu/daviscw}

\author[T.~Martin]{Taylor Martin}
\address{Department of Mathematics and Statistics, Sam Houston State University}
\email{taylor.martin@shsu.edu}
\urladdr{shsu.edu/academics/mathematics-and-statistics/faculty/martin.html}

\author[K.~Vance]{Katherine Vance}
\address{Department of Mathematics and Computer Science, Simpson College}
\email{katherine.vance@simpson.edu}
\urladdr{simpson.edu/faculty-staff/katherine-vance/}

\date{}

\begin{abstract}
In 2000, Habiro introduced the notion of \(C_k\)-equivalence of knots and links.  This geometric filtration is closely connected to finite type invariants, a class of invariants including Milnor's invariants.  Shortly thereafter, Ohyama, Taniyama, and Yamada proved that \(C_k\)-equivalence, and by extension finite type invariants, say very little about the unknotting number by showing that any knot is at most one crossing change away from being \(C_k\)-trivial for any \(k\in \mathbb{N}\).  The same is not true for links, since the pairwise linking number gives a lower bound on unlinking and is an invariant of \(C_2\)-equivalence.  We prove that, aside from the linking number, the result of Ohyama, Taniyama, and Yamada extends to links:  any \(n\)-component link with linking number zero can be reduced to a \(C_k\)-trivial link in at most \(n^2\) crossing changes.  As a consequence, Milnor's invariants carry only limited information about the unlinking number. To establish a lower bound, we  produce a sequence of \(n\)-component links for which the crossing change distance to a \(C_k\)-trivial link grows quadratically in \(n\). Notably, these bounds are independent of the choice of \(k\in \mathbb{N}\). Finally, we determine the exact number of crossing changes to a \(C_k\)-trivial link for links with nonzero linking numbers and where no component is \(C_k\)-trivial.
 
\end{abstract}

\maketitle
\section{Introduction}


The unlinking number $u(L)$ of a link $L$ in $S^3$ is the minimal number of crossing changes needed to transform $L$ into the unlink.  It has been the subject of intense study \cite{KM86, Kohn91, Kohn93,Likorish82, Yasutaka81}.  
  In 
\cite{KT} it is proven that the problem of computing the unlinking number is NP-hard.  

The most obvious lower bound on $u(L)$ comes from pairwise linking numbers.  A crossing change between distinct components of a link changes the pairwise linking number between those components by exactly $\pm 1$.  Thus, if $\Lambda(L)=\sum_{i<j}| \lk(L_i, L_j)|$, then $\Lambda(L)\le u(L)$.  

Two links are called link homotopic if one can be transformed to the other by only self-crossing changes \cite{M1}. (A self-crossing change is the act of changing a crossing between two arcs of a single component of $L$.)  In \cite{BDMOV2025}, the authors and Otto show that modulo link homotopy, $\Lambda(L)$ tells you nearly everything about the unlinking number. Indeed, if we set $\Lambda^*(L) = \sum_{i<j} \alpha_{ij}$ where $\alpha_{ij} = |\lk(L_i,L_j)|$ unless $i+1<j$ and $\lk(L_i,L_j)=0$, in which case $\alpha_{ij}=2$, then \cite{BDMOV2025} proves the following bounds.

\begin{theorem*}[Theorem 4.4 of \cite{BDMOV2025}]
If $L$ is an $n$-component link and $n_h(L)$ is the minimal number of crossing changes needed to replace $L$ by a homotopy trivial link 
then $\Lambda(L)\le n_h(L)\le \Lambda^*(L)$.
\end{theorem*}

In the 1950's \cite{M1, M2}, Milnor defined a collection of higher order linking invariants, now called Milnor invariants. These associate to an $n$-component link $L$ and a $k$-tuple $I\in \{1,\dots,n\}^k$ an integer $\mu_I(L)$. The value $k=|I|$ is called the length of the Milnor invariant.  While only the shortest length nonvanishing Milnor invariants are well defined,  Milnor showed that these invariants determine when a link is link homotopic to the unlink.

\begin{theorem*}[\cite{M1}]
An $n$-component link $L$ is link homotopic to the unlink if and only if $\mu_I(L)=0$ for all $I$ without repeated indices.  
\end{theorem*}

As a consequence, \cite[Theorem 4.4]{BDMOV2025} says that those higher order Milnor invariants with no repeated indices say very little about the unlinking number.  In this paper, we ask about Milnor invariants with repeated indices and bounded length.  

\begin{question}\label{quest: main}
Let $L$ be a link with vanishing pairwise linking numbers and let $k\in \N$. How many crossing changes does it take to replace $L$ by a link whose Milnor invariants of length up to $k$ vanish?  
\end{question} 

In \cite{Hab1}, Habiro introduced the language of \emph{claspers} to discuss surgery instructions which result in tying a link into some iterated Whitehead doubles of the Hopf link.  See Figure~\ref{fig: clasper surgery} for an example.  If two links are related by a sequence of surgeries along claspers of degree $k$, then they are called $C_k$-equivalent.  A link that is $C_k$-equivalent to the unlink will be called $C_k$-trivial.  It is an open question whether two links that are $C_k$-equivalent for all $k$ must be isotopic.

\begin{figure}
     \centering
     \begin{subfigure}[t]{0.3\textwidth}
         \centering
         \begin{tikzpicture}
         \node at (0,0){\includegraphics[width=.9\textwidth]{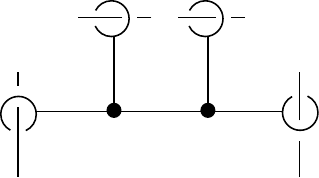}};
         \end{tikzpicture}
         \caption{A degree 3 clasper for a link}
     \end{subfigure}
     \hspace{0.1\textwidth}
     \begin{subfigure}[t]{0.3\textwidth}
     \centering
         \begin{tikzpicture}
         \node at (0,0){\includegraphics[width=.9\textwidth]{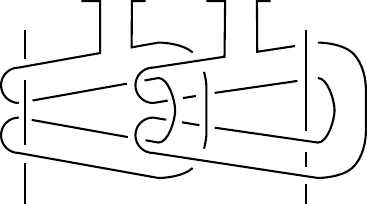}};
         \end{tikzpicture}
         \caption{the result of surgery along that clasper}
     \end{subfigure}
     \caption{}\label{fig: clasper surgery}
     \end{figure}

Milnor invariants of length $k$ are invariants of $C_{k}$-equivalence (Theorem 7.2 in \cite{Hab1}).  Thus, if one wants to understand how many crossing changes it takes to transform a link to a new link with all Milnor invariants vanishing, it makes sense to ask how many crossing changes it takes to transform $L$ into a $C_k$-trivial link.  We call this the $C_k$-trivializing number and denote it $u_k(L)$.  

For any knot $J$, there is another knot $J'$ which is $C_k$-equivalent to $J$ and has unknotting number $1$ \cite{Ohyama2000, OhTaYa2002}.  In other words, $u_k(J)\le 1$ for every knot $J$. 
We extend this result to links.  When all of the pairwise linking numbers vanish we get the following bound on $u_k(L)$ which grows quadratically with the number of components.

\begin{theorem}\label{thm:main lk=0}
Let $L=L_1\cup\dots\cup L_n$ be a link with pairwise linking numbers 0. For any $k\in \N$, $u_k(L)\le
n^2$.  Moreover, there is a sequence of crossing changes on $L$ consisting of 
\begin{enumerate}
\item at most one crossing change between two arcs on $L_i$ for each $i$ and 
\item at most two crossing changes between $L_i$ and $L_j$ for each $i\neq j$
\end{enumerate}
after which the resulting link is $C_k$-trivial.
\end{theorem}

 As linking numbers are invariants of $C_2$-equivalence, one does not expect a link with large pairwise linking numbers to be reducible to a $C_k$-trivial link in a small number of crossing changes; however, when all linking numbers are $\pm1$, we get better bounds:

\begin{restatable}{theorem}{MainThmLkOne}
\label{thm:main lk=1}
Let $L=L_1\cup\dots\cup L_n$ be a link with $|\lk(L_i,L_j)|=1$ for all $i,j$.  Then ${n\choose 2}\le u_k(L)\le 
{n+1\choose2}$.  Moreover, for all $i,j$ and for any $k\in \N$, there is a sequence of crossing changes transforming $L$ to a $C_k$-trivial link and consisting of 
\begin{enumerate}
\item at most one crossing change between two arcs on $L_i$ for each $i$ and 
\item one crossing change between $L_i$ and $L_j$ for each $i\neq j$.
\end{enumerate}
\end{restatable}


The bound ${n\choose 2}\le u_k(L)$ simply comes from the fact that $\Lambda(L)\le u_k(L)$.  We can get similar bounds for any arrangement of linking numbers. Indeed, both of the prior theorems are consequences of the following:

\begin{theorem}\label{thm:main abitrary linking}
Let $L=L_1\cup\dots\cup L_n$ be a link and define $\Lambda'(L) = \Sum_{1\le i<j\leq n}\left| | \lk(L_i,L_j)|-1\right|$.  For any $k\in \N$, $\Lambda(L)\le u_k(L)\le {n+1\choose 2}+\Lambda'(L)$.
\end{theorem}

We take a moment to motivate and demystify  $\Lambda'(L)$.  First note that $\Lambda'(L)$ records precisely the number of crossing changes needed to get a link to which Theorem~\ref{thm:main lk=1} applies.  Next, since  $\left|| \lk(L_i,L_j)|-1\right|\leq |\lk(L_i,L_j)|+1$, it follows that $ {n+1\choose 2}+\Lambda'(L)\leq \Lambda(L)+{n+1\choose 2}+{n\choose 2} = \Lambda(L)+n^2$. Thus, we arrive at a weaker but easier to parse result.

\begin{corollary}
    Let $L=L_1\cup\dots\cup L_n$ be a link. For any $k\in \N$, $\Lambda(L)\le u_k(L)\le n^2+\Lambda(L)$.
\end{corollary}

In the case that $\lk(L_i,L_j)\neq 0$, for all $i\neq j$, then $||\lk(L_i,L_j)|-1|=|\lk(L_i,L_j)|-1$, so 
\begin{eqnarray*}
u_k(L)&\leq& {n+1\choose 2}+\Sum_{1\leq i<j\leq n}||\lk(L_i,L_j)|-1|
\\&=& {n\choose 2}+n+\sum_{1\leq i<j\leq n}|\lk(L_i,L_j)|-\sum_{1\leq i<j\leq n}1\\
&=& n+\Lambda(L)
\end{eqnarray*}
This becomes particularly noteworthy if additionally every component of $L$ is not $C_k$-trivial.  In this case, any sequence of crossing changes resulting in a $C_k$-trivial link must include a self-crossing change on each component and must also include enough crossing changes between different components to eliminate all of the linking numbers.  

\begin{corollary}\label{cor:All linking nonzero}
    Let $L=L_1\cup\dots\cup L_n$ be a link such that for all $i,j$, $lk(L_i,L_j)\neq 0$ and $L_i$ is not $C_k$-trivial.  Then $u_k(L)= n+\Lambda(L)$.
\end{corollary}

$C_k$-equivalence is related to 4-dimensional topology in a few ways.  In \cite{CST}, Conant-Schneiderman-Teichner introduce Whitney tower concordance by asking when links cobound immersed annuli in $S^3\times[0,1]$ which extend to degree $k$ Whitney towers.  (Informally, a Whitney tower is an object inspired by the failure of the Whitney trick in dimension 4.  It is an immersed surface $S$ together with a collection of immersed Whitney disks, each pairing intersections between $S$ and lower degree disks.)  In \cite{CST2025}, they relate this to $C_k$-concordance, the equivalence relation generated by concordance and $C_k$-equivalence, first defined and studied in \cite[Section 5]{MeiYas2010}.
More precisely, they prove that $L$ is $C_k$-concordant to the unlink if and only if $L$ bounds an immersed disk which extends to a degree $k$ Whitney tower.

Tracing out a sequence of crossing changes yields an immersed union of annuli in $S^3\times[0,1]$ with one point of intersection for each crossing change.  By doing this for a sequence of ${n+1\choose 2}+\Lambda'(L)$ crossing changes from a link $L$ to a $C_{k+1}$-trivial link, and then stacking these immersed annuli atop an order $k+1$ Whitney tower, we get the following corollary.

\begin{corollary}
Let $L$ be an $n$-component link and $k\in\N$.   Then $L$ bounds a union of immersed disks so that all but ${n\choose 2}+\Lambda'(L)$ of the intersections and self intersections of these disks are paired by a degree $k$ Whitney tower.  
\end{corollary}

\subsection{Comparison with the work of Ohyama, Taniyama and Yamada on knots}

This work extends results of Ohyama, Taniyama, and Yamada from the setting of knots to links.  We take a moment an compare our strategy with theirs.  In \cite{OhTaYa2002}, Ohyama, Taniyama, and Yamada explore the fact that $C_n$-equivalence is generated by band sums with certain Brunnian links that they call special $C_n$ or $C_n'$ link models.    They then use that $C_1$-equivalence is trivial (for knots) as a base case to conclude that every knot is one crossing change from being $C_1$-trivial.  They then proceed inductively.  Up to $C_{n+1}$-equivalence, it does not matter how the bands use to sum on a $C_n$ model link with the underlying knot and with each other.  As a consequence they arrange so that all of these bands sums lie in a particular configuration and can be undone with a single crossing change.  

Our own arguments follow a similar logic, $C_1$-equivalence is determined by pairwise linking number, we then proceed inductively.
However, instead of using the structure of $C_n$-link models, we follow the language of clasper surgery developed by Habiro \cite{Hab1}.  Moreover, instead of undoing special $C_n$ link models, we explain how to replace a crossing change and a clasper of degree $k$ by a different crossing change and claspers of a greater degree.  This setting allows us access to the tools of Habiro which describe the structure needed to simplify clasper presentations.

\subsection{Are these bounds sharp?}

As we have seen in Corollary~\ref{cor:All linking nonzero}, in the presence of non-zero pairwise linking numbers, $u_k(L)$ is sometimes easy to determine, but if all pairwise linking numbers vanish, $u_k(L)$ is harder to determine.  By Theorem~\ref{thm:main lk=0} for any  $n$-component link $L$ with vanishing pairwise linking numbers and any $k\in \N$, $u_k(L)\le n^2$.  Is this bound sharp, or can one find a better upper bound via another analysis?  More precisely, for any $n$ and any $k$ define $C(n,k)$ to be the maximum value of $u_k(L)$ where $L$ is taken to be any $n$-component link with vanishing pairwise linking numbers.  

\begin{question}
Let \(n,k\in \N\).  Is $C(n,k)=n^2$?
\end{question}

In Section~\ref{sect: lower bound} we present some lower bounds on $C(n,k)$.  First, $C_1$-equivalence is generated by crossing changes, and  $C_2$-equivalence is generated by the $\Delta$ move of Figure~\ref{subfig: Delta move}.  Both are unlinking operations for links with vanishing pairwise linking numbers \cite{MN89}, so $C(n,k)=0$ if $k=1,2$.  Next, by \cite{Ohyama2000, OhTaYa2002}, any knot can be reduced to a $C_k$-trivial knot in one crossing change, so $C(1,k)=1$ for all $k>2$.  We prove the following, where the upper bounds come from Theorem~\ref{thm:main lk=0}.

\begin{theorem}
For any $k\geq4$, we have the following:
\begin{itemize}
\item (Theorem~\ref{thm: c2>=3}) $3\le C(2,k)\le 4$.
\item (Theorem~\ref{thm: c3>=6}) $5\le C(3,3)\le 9$ and $6\le C(3,k)\le 9$.
\item (Theorem~\ref{thm cn lower bound}) $2\ceil{\frac{n(n-2)}{3}}+n\le C(n,k)\le n^2$ when $n\ge 3$.
\end{itemize}
\end{theorem}

The final bullet point above roughly says $\frac{2}{3}n^2\le C(n,k)\le n^2$.  

All of this suggests that computing $u_k(L)$ may be a much more tractable problem than determining the actual unlinking number.  

\begin{question}
Is there an algorithm which determines $u_k(L)$? And is there a polynomial-time algorithm?
\end{question}

\subsection{Does the collection of all Milnor invariants tell you more?}

Our main result concludes that if you fix some $k\in \N$, then other than the pairwise linking numbers, Milnor's invariants of length up to $k$ (or any invariants of $C_k$-equivalence) tell you very little about the unlinking number.  If you try to study all Milnor invariants at once, then this question becomes harder for a few reasons.  First, these invariants are not well defined after the first non-vanishing.  Their indeterminacy has been closely studied in e.g.~\cite{HL1, HL2, KiMi2023, DNOP2020,AAD20}.  Secondly, the collection of all Milnor invariants is not an invariant of $C_k$-equivalence for any fixed $k\in\N$, so the strategy proposed in this paper cannot work.  

\begin{question}
Is there a bound $C_n\in \N$ so that every $n$-component link $L$ with vanishing pairwise linking numbers can be changed to a link with vanishing Milnor invariants in at most $C_n$ crossing changes?
\end{question}

The most famous class of links with vanishing Milnor invariants is that of boundary links.  These are links whose components bound pairwise disjoint Seifert surfaces.  Thus, it makes sense to ask a more geometric version of this question, 

\begin{question}
Is there some $C_n\in \N$ so that every $n$-component link $L$ with vanishing pairwise linking numbers can be changed to a boundary link in at most $C_n$ crossing changes?
\end{question}

When $n=2$, this question is resolved in \cite{AADG19}.  Any 2-component link with vanishing linking numbers is two crossing changes from being a boundary link. 

\subsection{Outline of the paper}

In Section~\ref{sect: background}, we recall  the theory of claspers and $C_k$ equivalence.  We close the section with a summary of a few of the moves used in \cite{Hab1} to produce a group structure on $C_k$-trivial links up to $C_{k+1}$ equivalence.  In Section~\ref{sect: proof of main theorem}, we present an inductive algorithm which produces for any link the desired sequence of crossing changes to get a $C_k$-trivial link.  Finally, in Section~\ref{sect: lower bound} we produce some links that cannot be reduced to $C_k$-trivial links in a small number of crossing changes.  

%

\section{Claspers, Clasper Surgery, and $C_k$-equivalence.}\label{sect: background}
We begin with an explanation of the theory.  A basic clasper $c$ is an embedded, twice-punctured disk in a 3-manifold  together with a decomposition of $c$ into two annuli joined by a band as in Figure~\ref{fig:simpleClasper}. 
This decomposition results in an identification of a regular neighborhood  of $c$ with a handlebody of genus 2 and a framed 2-component link $L_c$ as in Figure~\ref{fig:simpleClasperHood}.  Surgery along $c$ is given by performing surgery along $L_c$.  In the special case that $c$ is embedded into a link complement and the two annuli extend to embedded disks which intersect the link transversely, then this surgery has the effect of passing the strands of the link which cross through one of these disks through the other, as in Figure~\ref{fig:simpleSurgery}

\begin{figure}
     \centering
     \begin{subfigure}[b]{0.3\textwidth}
         \centering
         \begin{tikzpicture}
         \node at (0,0){\includegraphics[width=.9\textwidth]{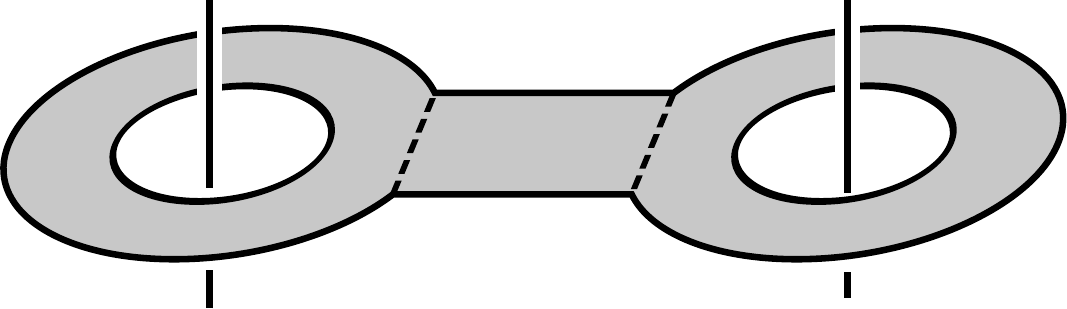}};
         \end{tikzpicture}
         \caption{}
         \label{fig:simpleClasper}
     \end{subfigure}
     \begin{subfigure}[b]{0.3\textwidth}
     \centering
         \begin{tikzpicture}
         \node at (0,0){\includegraphics[width=.9\textwidth]{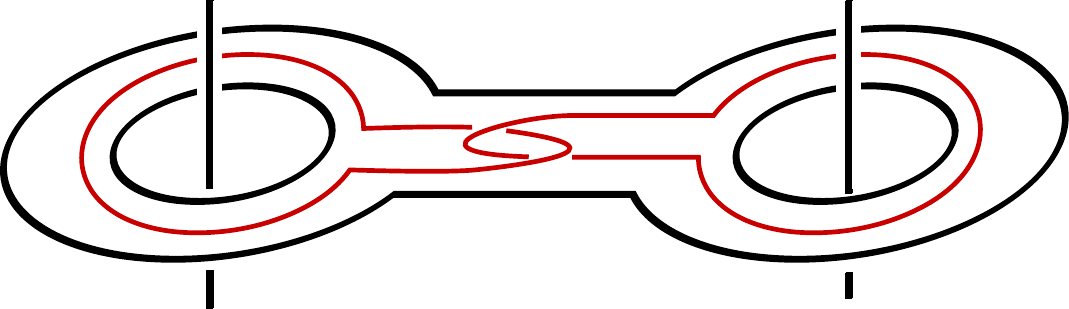}};
         \end{tikzpicture}
         \caption{}
         \label{fig:simpleClasperHood}
     \end{subfigure}
      \begin{subfigure}[b]{0.3\textwidth}
     \centering
         \begin{tikzpicture}
         \node at (0,0){\includegraphics[width=.7\textwidth]{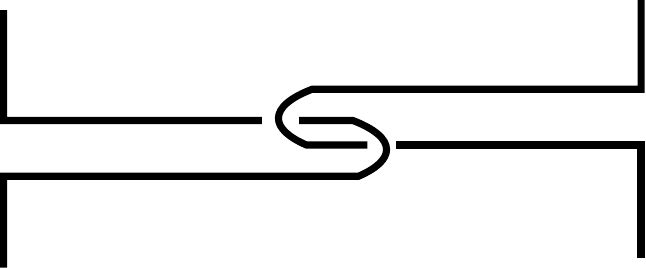}};
         \end{tikzpicture}
         \caption{}
         \label{fig:simpleSurgery}
     \end{subfigure}

        \caption{(a) A basic clasper, $C$, in the complement of a link $J$.  (b) A 2-component link sitting in a neighborhood of $C$ which inherits the blackboard framing. (c) Modifying $J$ by surgery along $C$.  
        }
        \label{fig: Simple clasper And Hood}
\end{figure}

A \emph{clasper} $c$ in a $3$--manifold $M$ 
is a compact, orientable, possibly disconnected surface embedded in $M$, equipped with a decomposition $c = A \cup B$.  See Figure~\ref{fig:clasperAcupB}.  Here $A$ is a disjoint union of subsurfaces, called \emph{constituents}, and consists of annular leaves, which are annuli, as well as disk leaves, nodes, and boxes, which are disks. (Note \cite{Hab1} refers to annular leaves as simply leaves.) $B$ is a disjoint union of bands, called \emph{edges}, connecting boundary arcs of constituents. 
The constituents and edges satisfy the following incidence relations: leaves and disk-leaves have one incident edge, 
nodes have three.  Boxes are incident to three edges, two of which are called inputs, and the other an output, as in Figure~\ref{fig: surgery box}.

\begin{figure}
     \centering
     \begin{subfigure}[b]{0.45\textwidth}
     \centering
        \begin{tikzpicture}
         \node at (0,0){\includegraphics[width=.9\textwidth]{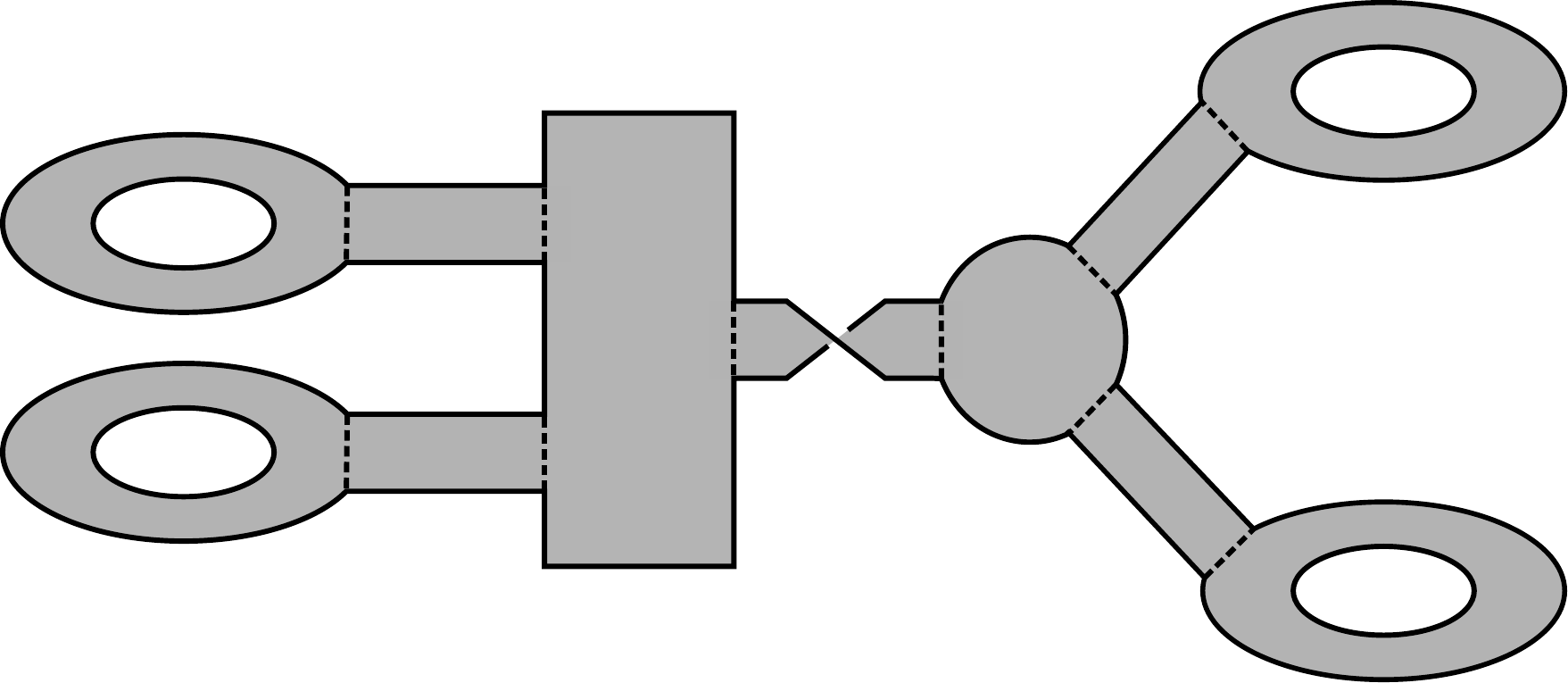}};
         \end{tikzpicture}
         \caption{}\label{fig:clasperAcupB}
     \end{subfigure}
     \begin{subfigure}[b]{0.45\textwidth}
     \centering
         \begin{tikzpicture}
         \node at (0,0){\includegraphics[width=.9\textwidth]{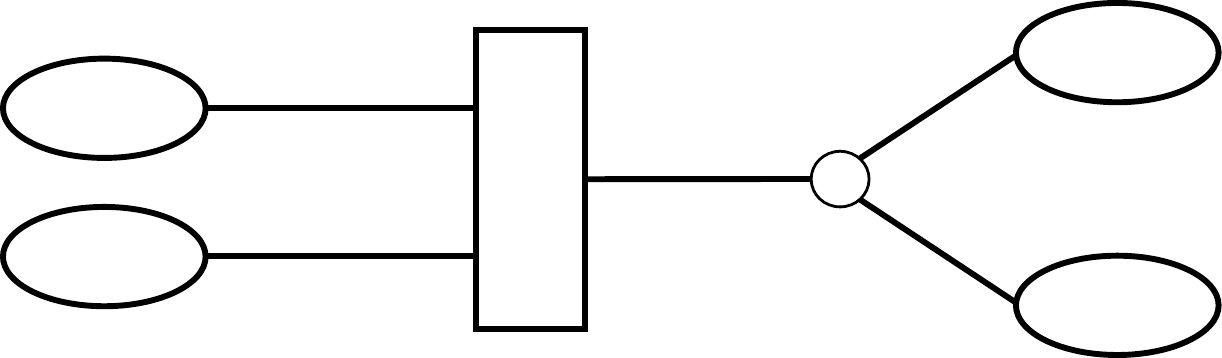}};
         \draw[fill=white, draw=black] (0.5,0) circle (.25);
         \node[] at (0.5,0) {$+$};
         \end{tikzpicture}
         \caption{}
         \label{fig:twistedClasper}
     \end{subfigure}
     \caption{(a) A clasper $c$ showing its decomposition into bands and constituents.  The middle band has a right handed half twist. The set of constituents includes one box, one node, and four leaves. (b) The same clasper, encoded as a union of leaves, nodes, boxes, and framed arcs between them. The ``\(+\)" indicates a positive half twist difference from the blackboard framing.  A ``\(-\)'' will be used for a negative half twist.}
\end{figure}

For each  leaf, $\ell$, we call the component of $\bdry \ell$ which intersects a band the outer boundary.  In the case of an annular leaf, the other boundary component is called the inner boundary.  

Following the convention of \cite{Hab1}, we will often encode a clasper as a union of nodes, boxes, framed closed curves representing its leaves, and framed arcs representing bands, as in Figure~\ref{fig:twistedClasper}.  The leaves will always have the blackboard framing.  When we do not specify a framing on a band we mean the blackboard framing.  When the framing on a band differs from the blackboard framing by a positive (or negative) half twist we will decorate it with a $+$ sign (or $-$ sign).

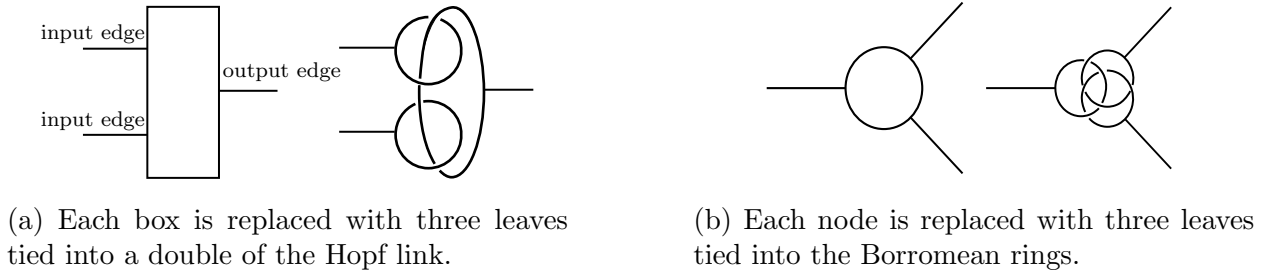
\begin{figure}
     \centering
     \begin{subfigure}[t]{0.45\textwidth}
     \centering
        \begin{tikzpicture}
           \node[left] at (0,0){\includegraphics[width=.35\textwidth]{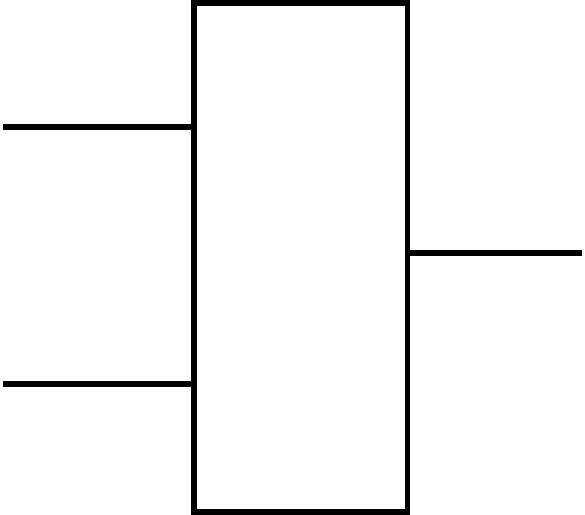}};
           \node[above] at (-2.6,.5) {\tiny{input edge}};
           \node[above] at (-2.6,-.65) {\tiny{input edge}};
           \node[above] at (-.1,0) {\tiny{output edge}};
         \node[right] at (.5,0){\includegraphics[width=.35\textwidth]{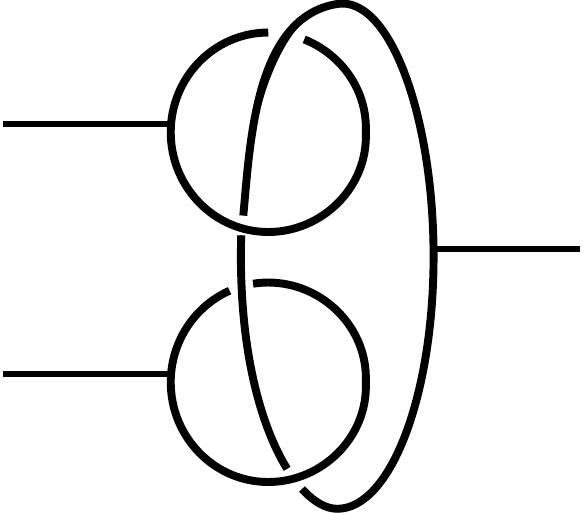}};;
         \end{tikzpicture}
         \caption{Each box is replaced with three leaves tied into a double of the Hopf link.}\label{fig: surgery box}
     \end{subfigure}
     \hfill
     \begin{subfigure}[t]{0.45\textwidth}
     \centering
         \begin{tikzpicture}
         \node[left] at (0,0){\includegraphics[width=.35\textwidth]{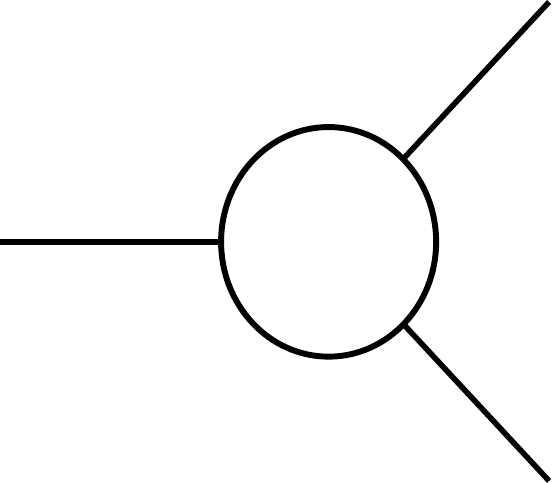}};
         \node[right] at (0,0){\includegraphics[width=.35\textwidth]{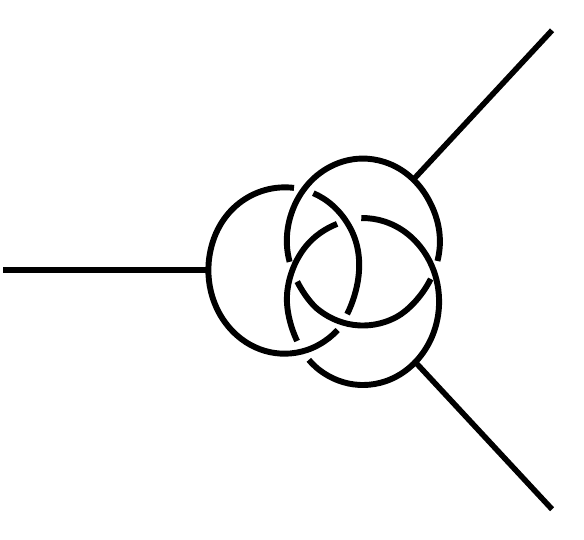}};
         \end{tikzpicture}
         \caption{Each node is replaced with three leaves tied into the Borromean rings.}
         \label{fig: surgery node}
     \end{subfigure}
     \caption{Transforming a clasper into a collection of simple claspers.}
\end{figure}

In order to perform surgery along a clasper we first replace each disk leaf with a collar neighborhood of its boundary, transforming it to an annular leaf, and then perform the move of Figure~\ref{fig: surgery box} at each box and the move of Figure~\ref{fig: surgery node} at each node. Finally we perform surgery along the resulting collection of basic claspers, {or rather, surgery on the canonically associated framed link.}   

If a clasper $c$ is connected, has no boxes, and $H_1(c)$ is carried by its leaves, then we call $c$ \emph{tree-shaped}.  If a tree-shaped clasper has at least one disk-leaf then it is called \emph{admissible}.  For such a clasper, a sequence of handle cancellations reveals that surgery preserves the diffeomorphism type of the 3-manifold $M$.  If every leaf is a disk leaf then $c$ is called \emph{strict}.

If $L$ is a link which intersects the interior of the disk leaves of a strict tree clasper $c$ but is otherwise disjoint from $c$, then $c$ is called a \emph{clasper for $L$} and this surgery changes $L$ by a band sum with a Brunnian link, as is exemplified in Figure~\ref{fig: clasper surgery}.  
A finite collection of disjoint trees is called a \emph{forest}.
We call a forest \emph{admissible} or \emph{strict} if its elements are admissible or strict respectively.  If a $c$ is a strict tree clasper for a link $L$ and $L$ intersects each of 
the disk-leaves of $c$ transversely in a single point then $c$ is \emph{simple}.  If every tree in a forest is simple then that forest is simple. 

The \emph{degree} of a tree-shaped clasper is defined to be the number of nodes in that clasper plus 1 (or equivalently, the number of leaves minus 1.)  The degree of a forest is the minimal degree among all of its trees.

If $L$ is a link, $c$ is a degree $k$ simple forest for $L$, and $L^c$ is the result of changing $L$ by surgery along $c$, then we say that $L$ and $L^c$ are $C_k$-equivalent.  

In \cite{Hab1}, Habiro presents 12 moves which preserve the result of surgery along a clasper.  We recall those moves most relevant to our analysis in Figure~\ref{fig: moves 6 and 12}. 

\begin{figure}
     \begin{subfigure}[t]{0.4\textwidth}
     \centering
         \begin{tikzpicture}
         \node at (0,0){\includegraphics[width=.6\textwidth]{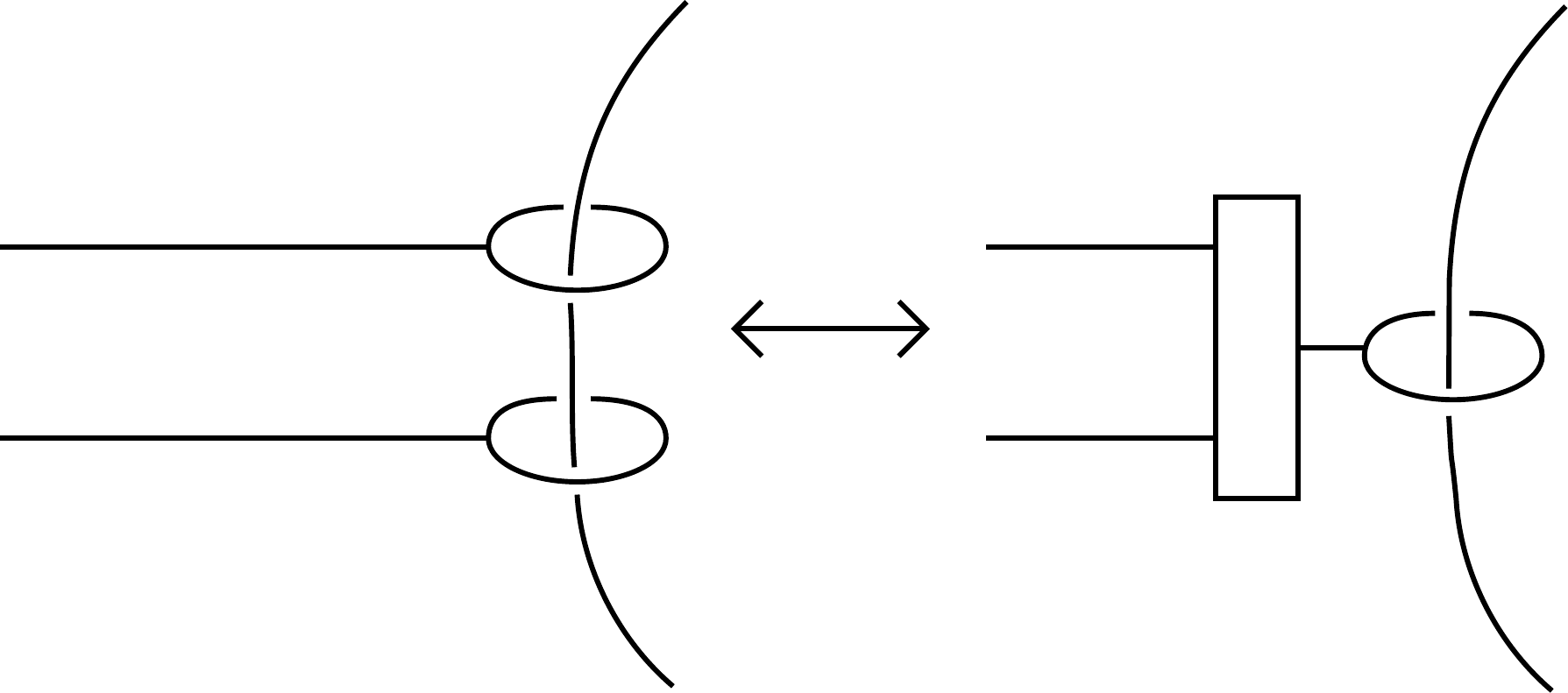}};
         \end{tikzpicture}
         \caption{Habiro's move 6
         }
         \label{subfig: move 6}
     \end{subfigure}
     \hfill
     \begin{subfigure}[t]{0.5\textwidth}
     \centering
         \begin{tikzpicture}
         \node at (0,0){\includegraphics[width=\textwidth]{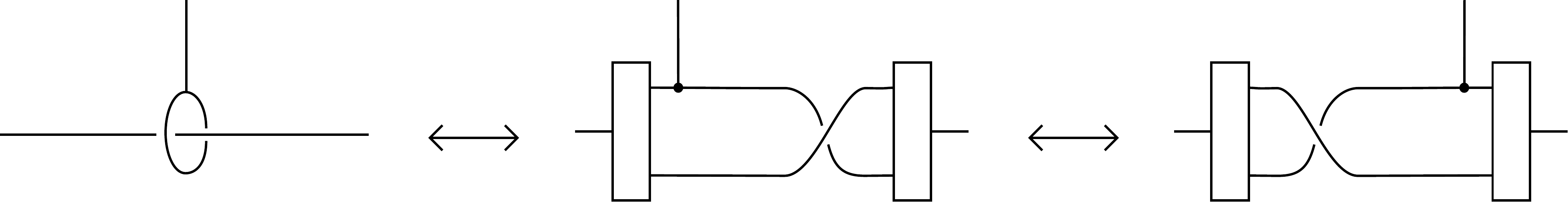}};
         \draw[fill=white, draw=black] (3.25,0.075) circle(.2);
         \node at (3.25,0.075) {$\scriptscriptstyle{-}$};
         \draw[fill=white, draw=black] (-0.175,0.075) circle(.2);
         \node at (-0.175,0.075) {$\scriptscriptstyle{-}$};
         \end{tikzpicture}
         \caption{Habiro's move 12}
         \label{subfig: move 12}
     \end{subfigure}
     \caption{Two moves which preserve the result of clasper surgery.}\label{fig: moves 6 and 12}
     \end{figure}
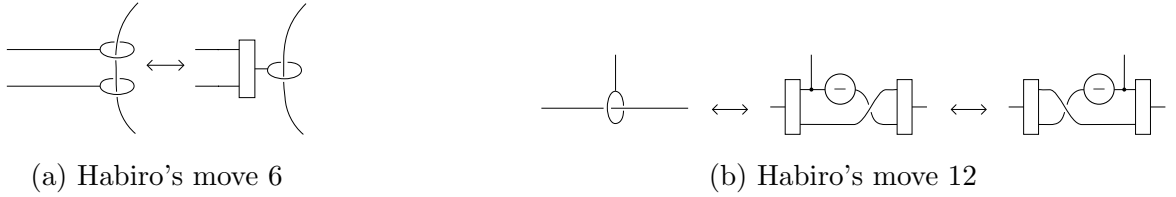

Habiro also studies the notion of homotopy of claspers in order to measure the difference between $C_k$-equivalence and $C_{k+1}$-equivalence.  We recall the main result here. Items (1) through (3) are used to produce a monoid structure on the set of $C_k$-trivial links up to $C_{k+1}$-equivalence \cite[Theorem 4.3]{Hab1}.  Item (4) is used to upgrade this to a group structure \cite[Theorem 4.7]{Hab1}.  

\begin{theorem}\label{thm: Habiro moves}
Let $L$ be a link and $F$ and $F'$ be disjoint simple forests for $L$.
\begin{enumerate}
\item \cite[Proposition 4.4]{Hab1} If $F$ and $F'$ differ by passing a disk-leaf on a degree $k$ tree of $F$ through a disk leaf on a different tree of degree $k'$,  then $L^F$ and $L^{F'}$ are related by a single simple $C_{k+k'}$ move.
\item \cite[Proposition 4.6]{Hab1} If $F$ and $F'$ differ by passing an edge in a degree $k$ tree through an edge on a degree $k'$ tree then $L^F$ and $L^{F'}$ are related by a single simple $C_{k+k'+1}$ move.
\item \cite[Proposition 4.5]{Hab1} If $F$ and $F'$ differ by passing an edge in a degree $k$ tree through $L$ then $L^F$ and $L^{F'}$ are related by a single simple $C_{k+1}$ move.
\item \cite[Theorem 4.7]{Hab1} If $F$ and $F'$ differ by introducing a (positive or negative) half twist along any two edges on a single tree of degree $k$, then $L^F$ and $L^{F'}$ are $C_{k+1}$-equivalent.
\end{enumerate}
\end{theorem}
\begin{proof}
Parts (1) through (3) are explicitly stated and proved in \cite{Hab1}.   Item (4) implicitly appears in the proof of \cite[Theorem 4.7]{Hab1}, so an assiduous reader is directed there for the proof.  At the risk of circular logic, we instead explain why (4) follows from \cite[Theorem 4.7]{Hab1}.  

\cite[Definition 4.2]{Hab1} gives an equivalence relation called homotopy for claspers for a fixed link $L$.  Let $\mathcal{F}_k(L)$ denote set of all homotopy classes of simple strict forest claspers of degree $k$. Under the operation of disjoint union,  $\mathcal{F}_k(L)$ is a monoid.  The map sending $F\in\mathcal{F}_k(L)$  to the class of $L^F$ up to $C_{k+1}$-equivalence is well defined \cite[Theorem 4.3]{Hab1}.

We will denote by $\widetilde{\mathcal{F}}_k(L)$  the quotient of $\mathcal{F}_k(L)$ given by setting $[T]+[T']=0$ whenever $T$ and $T'$ are simple tree claspers differing by a 
single half twist along one edge.  According to \cite[Theorem 4.7]{Hab1} the map sending $F\in \widetilde{\mathcal{F}}_k(L)$ to $L^F$ up to $C_{k+1}$-equivalence is still well defined.

Now let $F$ and $F'$ differ by adding a half twist along two edges $e_1$ and $e_2$ of a single tree $T\subseteq F$ resulting in $T'\subseteq F'$.  Let $F_1$ be the result of adding a half twist along only $e_1$ on $T$ resulting in $T_1\subseteq F_1$.  
Let $F_0=F\setminus T$  and $L_0=L^{F_0}$.  We can now treat $[T]$, $[T']$ and $[T_1]$ as elements of $\widetilde{\F}_l(L_0)$.    Notice that as $T_1$ differs from $T$ by adding a half twist along $e_1$ and $T'$ differs from $T_1$ by adding a half twist along $e_2$, it follows that both $[T]$ and $[T']$ are additive inverses to $[T_1]$ in $\widetilde{F}_k(L_0)$ and thus $[T]=[T']$.  Thus, up to $C_{k+1}$-equivalence, 
$$L^F = (L_0)^T=(L_0)^{T'}=L^{F'},$$
completing the proof.
\end{proof}

\section{Crossing changes and raising degree.}\label{sect: proof of main theorem}

In this section we prove Theorems~\ref{thm:main lk=0}, \ref{thm:main lk=1}, and \ref{thm:main abitrary linking}.  Theorem~\ref{thm:main lk=0} follows as the special case of Theorem~\ref{thm:main abitrary linking} where all of the linking numbers vanish.

For any link $L$, after $\Lambda'(L) = \Sum_{i<j}||\lk(L_i,L_j)|-1|$ crossing changes, $L$ is reduced to a link whose every linking number is $\pm1$, so an appeal to Theorem \ref{thm:main lk=1} will now complete the proof of Theorem~\ref{thm:main abitrary linking}. Thus, it suffices to provide a proof of Theorem \ref{thm:main lk=1}.


\MainThmLkOne*

\begin{proof}[Proof of Theorem \ref{thm:main lk=1}.]

We proceed inductively on $k$.  Let $L$ be a link with $|\lk(L_i,L_j)|=1$ for all $i,j$.  As shown in \cite{MN89}, any two links with identical pairwise linking numbers are related by a sequence of $\Delta$-moves (Figure~\ref{subfig: Delta move}), which in turn are realized by surgery on degree 2 trees (Figure~\ref{subfig: Delta move clasper}). Thus, if we make the ${n\choose 2}$ crossing changes needed to kill the linking numbers, we arrive at a $C_2$-trivial link. In order to start the induction, we also make the trivial crossing change corresponding to the degree 1 simple clasper of Figure~\ref{subfig: trivial clasper} on each component of $L$.

\begin{figure}
     \centering
     \begin{subfigure}[t]{0.3\textwidth}
         \centering
         \begin{tikzpicture}
         \node at (0,0){\includegraphics[width=\textwidth]{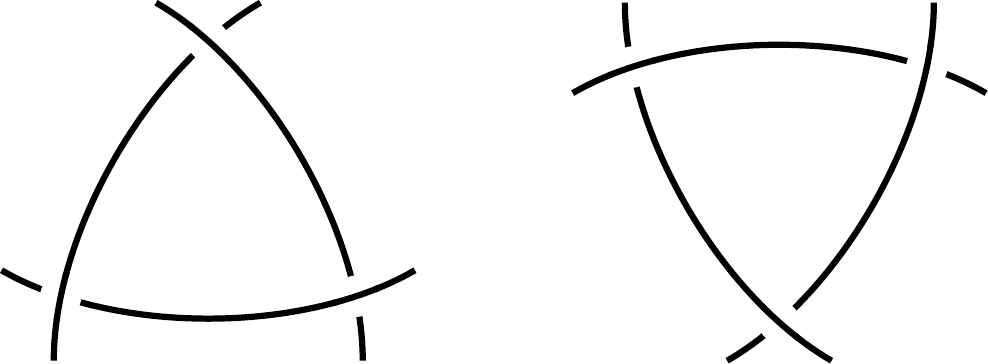}};
         \node at (0,-0.2) {$\longrightarrow$};
         \node at (0,0.2) {\tiny{$\Delta$}};
         \end{tikzpicture}
         \caption{The Delta move}
         \label{subfig: Delta move}
     \end{subfigure}
     \hfill
     \begin{subfigure}[t]{0.3\textwidth}
     \centering
         \begin{tikzpicture}
         \node at (0,0){\includegraphics[width=.5\textwidth]{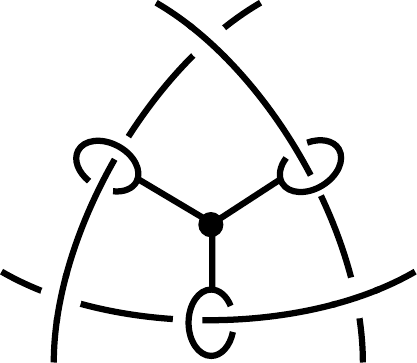}};
         \end{tikzpicture}
         \caption{A degree 2 tree that acts by the Delta move. 
         }
         \label{subfig: Delta move clasper}
     \end{subfigure}
     \hfill
     \begin{subfigure}[t]{0.3\textwidth}
     \centering
         \begin{tikzpicture}
         \node at (0,0){\includegraphics[width=.5\textwidth]{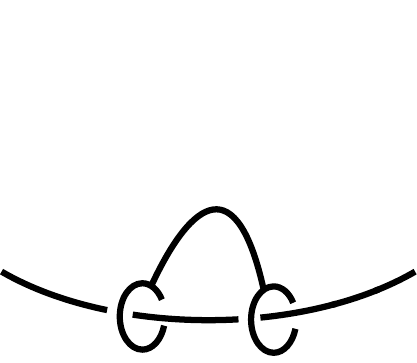}};
         \end{tikzpicture}
         \caption{A simple clasper that changes $L$ by two Reidemeister 1 moves.
         }
         \label{subfig: trivial clasper}
     \end{subfigure}
     \caption{}\label{}
     \end{figure}

Let $L$ be an $n$-component link with all pairwise linking numbers equal to $\pm1$ and $k\ge 2$.  For the sake of induction, we assume  that $L$ can be reduced to a $C_k$-trivial link in one self-crossing change on each component, and one crossing change between each pair of components.  By realizing each of these crossing changes as surgery along a degree 1 simple clasper, we see that there is a simple forest $F$ for $L$ so that $L^F$ is the unlink and for which \(F=\{c_{ij}\mid1\le i\le j\le n\}\cup F^k\) where
\begin{itemize}
\item  Each $c_{ij}$ is a degree 1 clasper whose disk leaves intersect $L_i$ and $L_j$ respectively each in one point
\item $F^k$ is a forest of degree at least $k$.
\end{itemize}

Let $T\in F^k$ be a tree of degree exactly $k$.  The proof will replace $T$ by possibly many new strict trees each of degree greater than $k$ and replace one $c_{ij}$ by a new degree 1 clasper, $c_{ij}'$, while preserving the result of surgery.

 A reasonably straightforward induction on the degree of $T$ shows that there must be a node $N$ of $T$ which shares an edge with each of two different leaves, $\ell_1$ and $\ell_2$.  Let $\ell_1$ intersect $L_i$ and $\ell_2$ intersect $L_j$. It may be that $i=j$. Let $a_1$ and $a_2$ be the paths in $L$ from $\ell_1\cap L_i$ to $c_{ij}\cap L_i$ and from $\ell_2\cap L_j$ to $c_{ij}\cap L_j$, respectively.  (If $i=j$ then we require that $a_1$ and $a_2$ be disjoint.)  As in Figure~\ref{subfig: Main theorem getting leaves close start} $a_1$ and $a_2$ might intersect other leaves in elements of $F_k$ (including other leaves of $T$), and simple claspers $c_{ik}$.

\begin{figure}
     \centering
     \begin{subfigure}[t]{0.3\textwidth}
         \centering
         \begin{tikzpicture}
         \node at (0,0){\includegraphics[width=.75\textwidth]{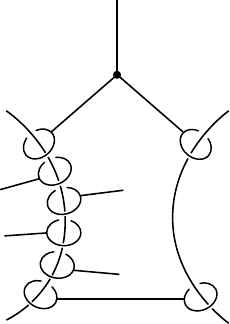}};
         \node at (-1.7,1.2) {$L_i$};
         \node at (1.8,1.2) {$L_j$};
         \node at (0,-2.5) {$c_{ij}$};
         \node at (-2.1,-.4) {$S$};
         \node at (-2.1,-1.1) {$c_{ik}$};
         \node at (0,1.1) {$N$};
         \node at (-0.3,2.5) {$T$};
         \node at (0.35,-0.45) {$T$};
         \node at (0.35,-1.85) {$T$};
         \end{tikzpicture}
         \caption{Leaves of $T$, leaves of simple claspers $c_{ik}$, and other elements  $S\in F_k$ intersecting the  interior of $a_i$.}
         \label{subfig: Main theorem getting leaves close start}
     \end{subfigure}
     \hfill
     \begin{subfigure}[t]{0.3\textwidth}
     \centering
         \begin{tikzpicture}
         \node at (0,0){\includegraphics[width=.75\textwidth]{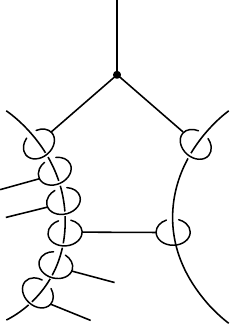}};
         \node at (0,1.1) {$N$};
         \node at (0,-1.4) {$c_{ij}$};
         \node at (-2.1,-.4) {$S$};
         \node at (-2.1,-1.1) {$c_{ik}$};
         \node at (-1.7,1.2) {$L_i$};
         \node at (1.8,1.2) {$L_j$};
         \node at (-0.3,2.5) {$T$};
         \node at (0.2,-2) {$T$};
         \node at (-0.2,-2.6) {$T$};
         \end{tikzpicture}
         \caption{Sliding leaves of $T$ past $c_{ij}$.}
         \label{subfig: Main theorem getting leaves close step 1}
     \end{subfigure}
     \hfill
     \begin{subfigure}[t]{0.3\textwidth}
     \centering
         \begin{tikzpicture}
         \node at (0,0){\includegraphics[width=.83\textwidth]{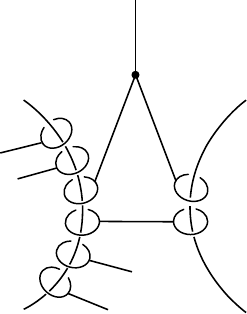}};
         \node at (-.1,1.4) {$N$};
         \node at (0.2,-1.35) {$c_{ij}$};
         \node at (-2.2,.2) {$S$};
         \node at (-2,-.6) {$c_{ik}$};
         \node at (-1.6,1.25) {$L_i$};
         \node at (2,1.25) {$L_j$};
         \node at (-0.1,2.5) {$T$};
         \node at (0.4,-2) {$T$};
         \node at (0,-2.6) {$T$};
         \end{tikzpicture}
         \caption{Sliding all other leaves past $\ell_1$.}
         \label{subfig: Main theorem getting leaves close done}
     \end{subfigure}
     \caption{}\label{fig: Main theorem getting leaves close}
     \end{figure}

According to Theorem~\ref{thm: Habiro moves} (1), at the cost of adding to $F$ trees of degree at least $k+1$, we can slide a degree $k$ leaf through a leaf of any other tree. Therefore, we may arrange that there are no leaves of $T$ in the interior of $a_1$ (or $a_2$) by sliding them past $c_{ij}$ and any other leaves of other trees between as in Figure~\ref{subfig: Main theorem getting leaves close step 1}.  Next, at the cost of adding trees of degree at least $k+1$, we can remove all leaves of other trees (either $S\in F_k$ where $S\neq T$ or simple claspers $c_{ik}$) in the interior of $a_1$ (or $a_2$) by sliding them past $\ell_1$ (respectively $\ell_2$), which is a leaf of the degree $k$ tree $T$, as in Figure~\ref{subfig: Main theorem getting leaves close done}.  We have now arranged that the arcs in $L$ from $\ell_1\cap L$ and $\ell_2\cap L$ to $c_{ij}\cap L$ contain no other points in $T\cap L$.
 
\begin{figure}
     \centering
     \begin{subfigure}[t]{0.22\textwidth}
         \centering
         \begin{tikzpicture}
         \node at (0,0){\includegraphics[width=.9\textwidth]{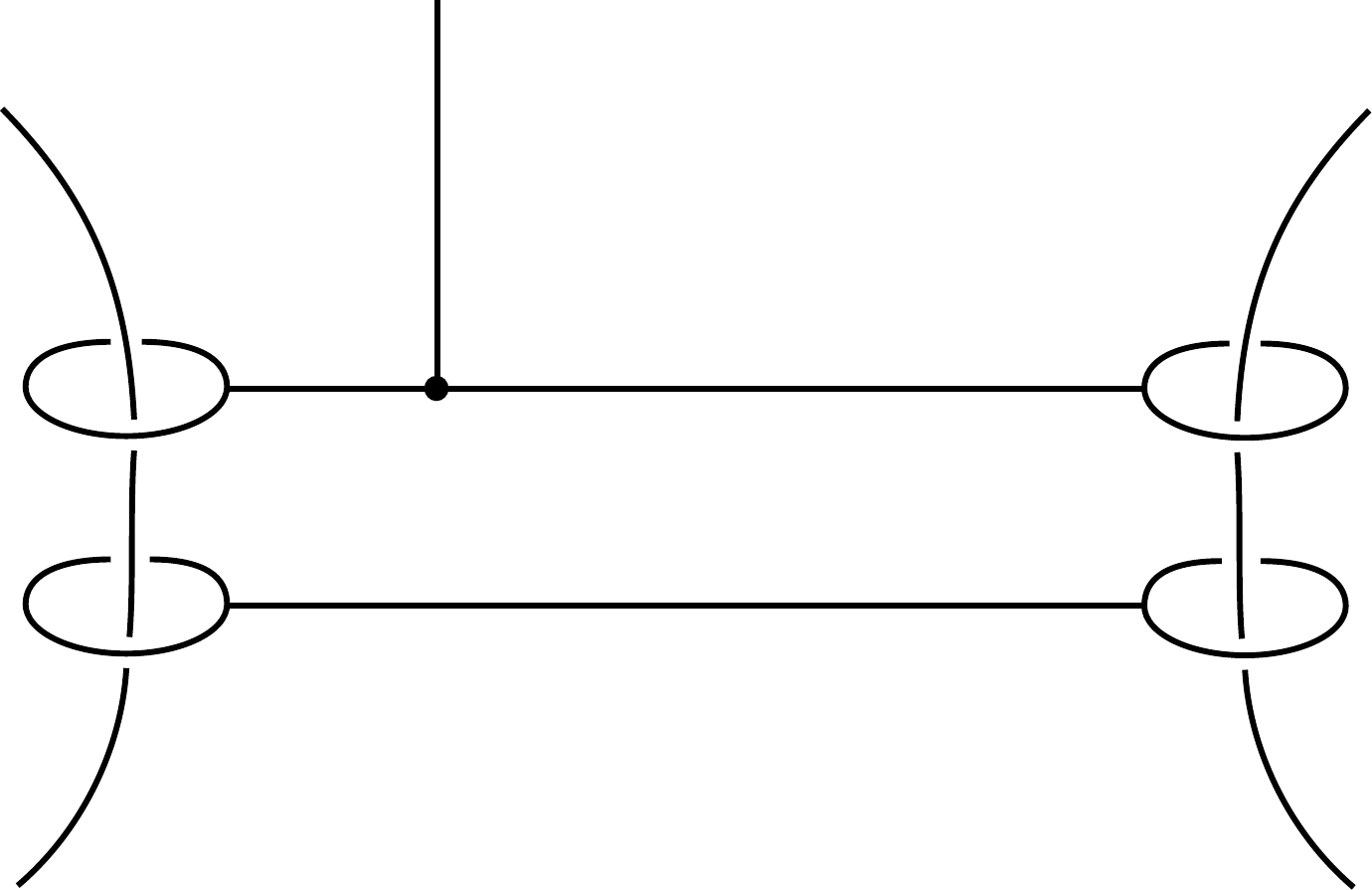}};
         \node at (0,-0.7) {$c_{ij}$};
         \node at (-1.75,0.15) {$\ell_1$};
         \node at (1.8,0.15) {$\ell_2$};
         \node at (-0.6,-0.1) {$N$};
         \end{tikzpicture}
         \caption{The arc in $T$ from $\ell_1$ to $N$ to $\ell_2$ is parallel to the edge of $c_{ij}$}
         \label{subfig: paths parallel}
     \end{subfigure}
     \hfill
     \begin{subfigure}[t]{0.22\textwidth}
     \centering
         \begin{tikzpicture}
         \node at (0,0){\includegraphics[width=.9\textwidth]{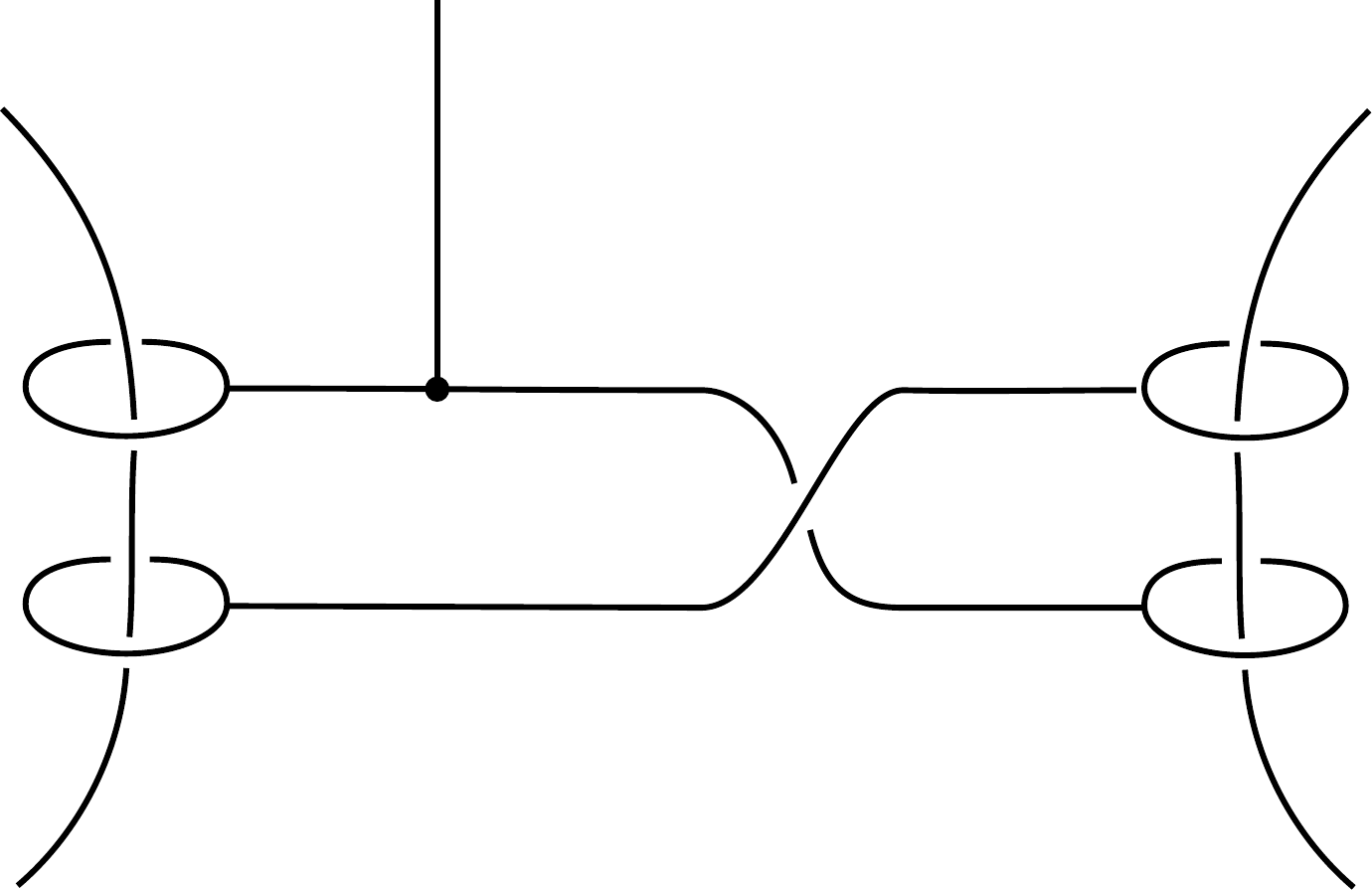}};
         \node at (-0.2,0.15) {
         ${\scriptscriptstyle \overline{s}}$};
         \draw[fill=white, draw=black] (-0.2,0.15) circle(.2);
         \node at (-0.2,.15) {$\scriptscriptstyle{-}$};
         \draw[fill=white, draw=black] (-0.6,.65) circle(.2);
         \node at (-0.6,.65) {$\scriptscriptstyle{-}$};
         \end{tikzpicture}
         \caption{Sliding $\ell_2$ past the leaf of $c_{ij}$ and adding a negative half twist to the arc in T from $N$ to $\ell_2$ and one more elsewhere in $T$.}
         \label{subfig: twisted paths}
     \end{subfigure}
     \hfill
     \begin{subfigure}[t]{0.22\textwidth}
     \centering
         \begin{tikzpicture}
         \node at (0,0){\includegraphics[width=.9\textwidth]{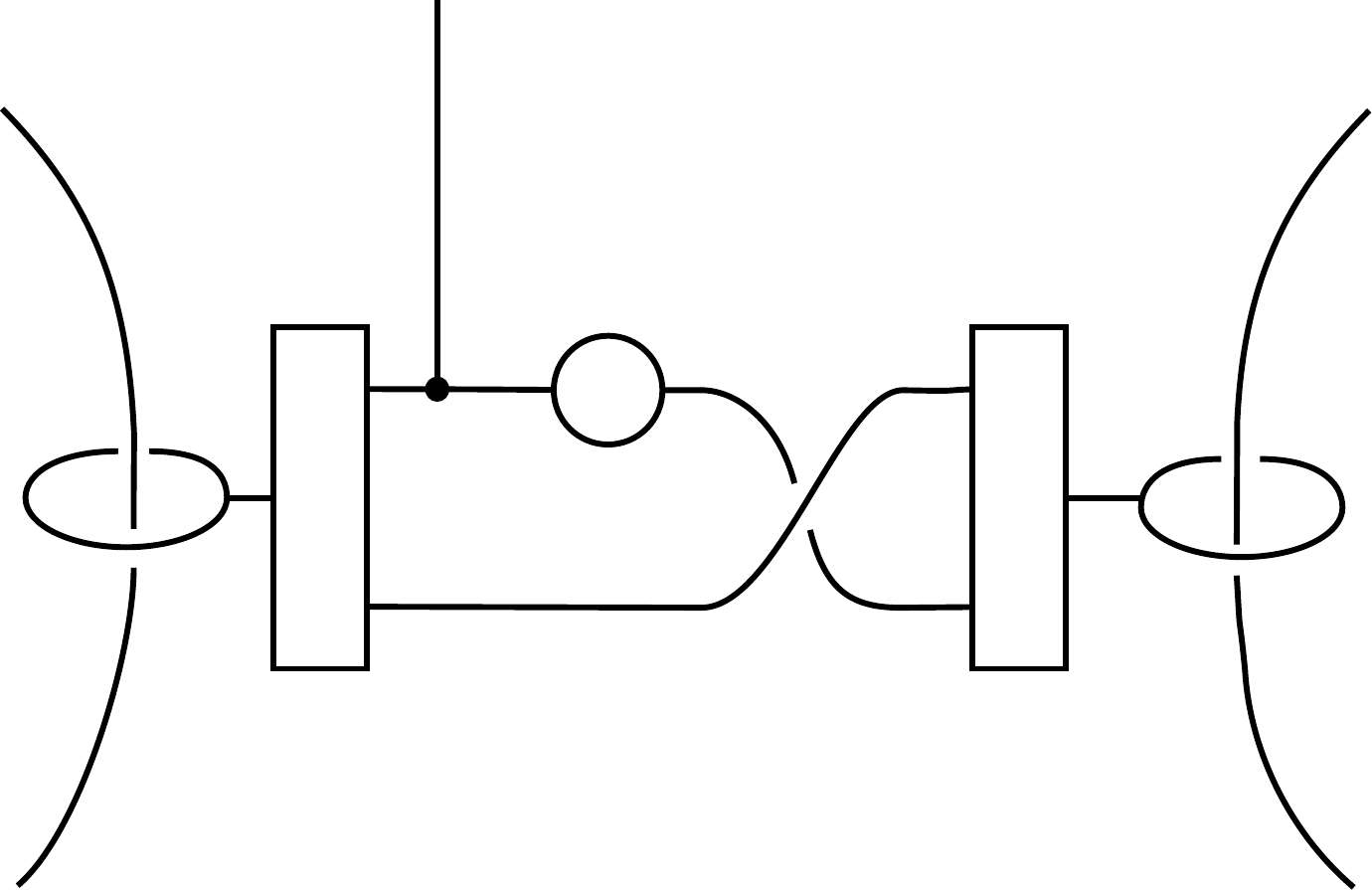}};
         \node at (-0.2,0.15) {
         $\scriptstyle{-}$};
         \draw[fill=white, draw=black] (-0.2,0.15) circle(.2);
         \node at (-0.2,.15) {$\scriptscriptstyle{-}$};
         \draw[fill=white, draw=black] (-0.6,.65) circle(.2);
         \node at (-0.6,.65) {$\scriptscriptstyle{-}$};
         \end{tikzpicture}
         \caption{Using Habiro's move 6 twice}
         \label{subfig: move 6 twice}
     \end{subfigure}
     \hfill
     \begin{subfigure}[t]{0.22\textwidth}
     \centering
         \begin{tikzpicture}
         \node at (0,0){\includegraphics[width=.9\textwidth]{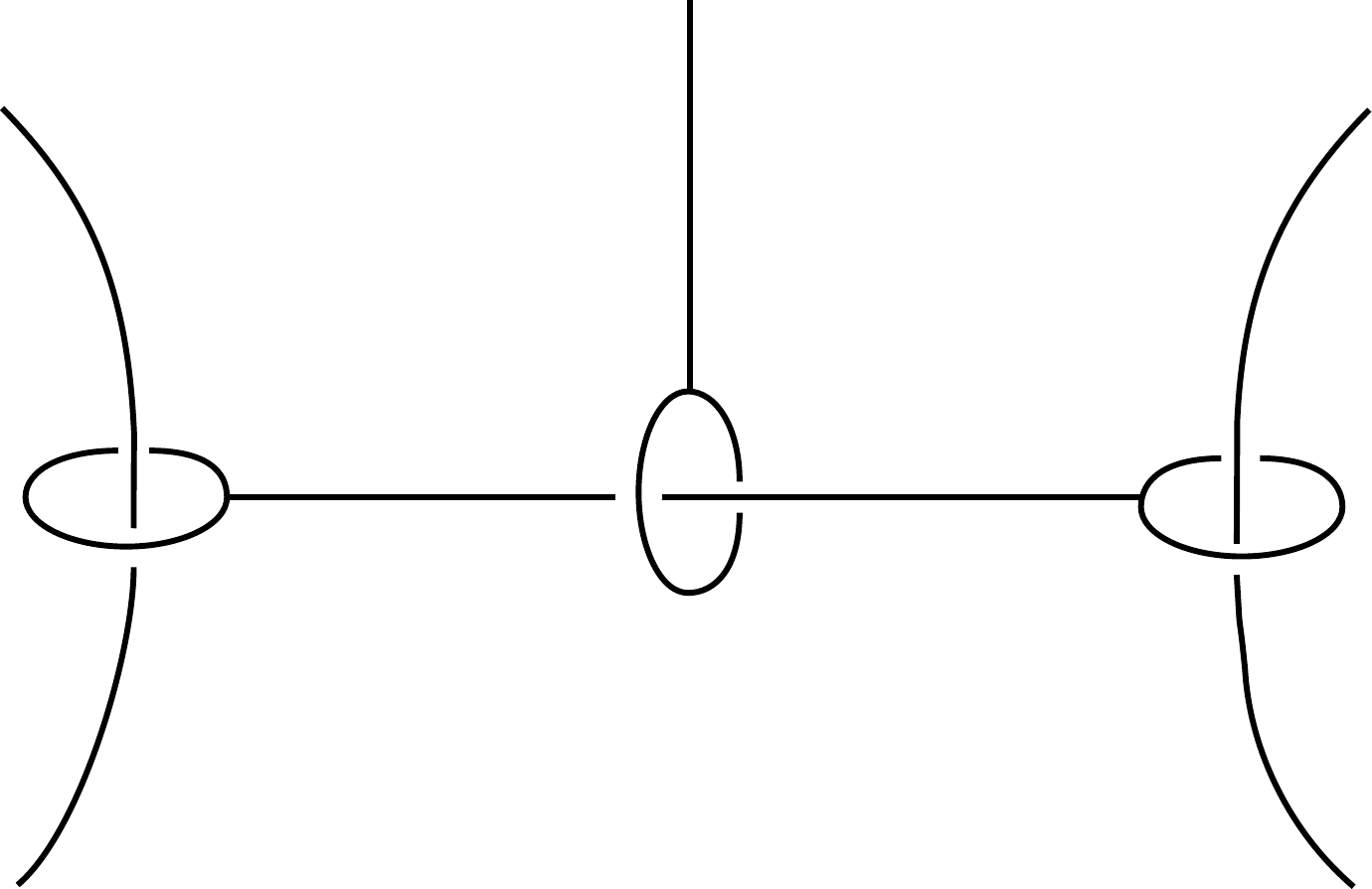}};
         \node at (-0.4,0.75) {$T'$};
         \node at (0.6,-0.5) {$c_{ij}$};
         \draw[fill=white, draw=black] (0,.65) circle(.2);
         \node at (0,.65) {$\scriptscriptstyle{-}$};
         \end{tikzpicture}
         \caption{Habiro's move 12}
         \label{subfig: move 12 applied}
     \end{subfigure}
     \caption{
     }\label{fig: Setup for Move 12}
     \end{figure}

We now introduce two new arcs in $F$.  The first is given by $e_{ij}$, the core of the band in $c_{ij}$.  The second, $\gamma$, is the concatenation of the paths in $T$ from $\ell_1$ to $N$ and from $N$ to $\ell_2$.  According to Theorem~\ref{thm: Habiro moves} (2) and (3), at the cost of adding to $F$ more trees of degree at least $k+1$  we may pass $e_{ij}$ through bands in $F$ 
and $\gamma$ through any simple clasper $c_{k,\ell}$ and any component of $L$.  By doing so we may arrange that $e_{ij}$ and $\gamma$ are parallel, as in Figure~\ref{subfig: paths parallel}.  If needed, we apply Theorem~\ref{thm: Habiro moves} (4) to introduce the needed half twist along with a half twist on some other edge of $T$  and slide $\ell_2$ past the leaf of $c_{ij}$ (at the cost of adding to $F$ a higher degree tree), as in Figure~\ref{subfig: twisted paths}.

We then apply Habiro's moves 6 and 12 (Figures~\ref{subfig: move 6 twice} and \ref{subfig: move 12 applied}, \cite[Proposition 2.7]{Hab1}) which shows that the result of surgery $L^F$ is preserved by replacing $T\cup c_{ij}$ by $T'\cup c_{ij}$, where one of the disk leaves of $T'$ intersects only $e_{ij}$.  But then performing surgery along $T'$ changes $L\cup e_{ij}$ by tying in a Brunnian link.  Up to isotopy, this preserves $L$ and changes $e_{ij}$ to a new degree 1 simple clasper $e_{ij}'$.  Thus, we may replace $F$ by $F\setminus \{T,e_{ij}\}\cup \{e_{ij'}\}$.  If we repeat this argument on every degree $k$ tree in $F$, then we now see a new simple forest $F'$ for $L$ consisting of the same number of degree 1 claspers, no trees of degree less than or equal to $k$, and an uncontrolled number of trees of degree greater than $k$, and so that $L^{F'}=L^F$, which by assumption was the unlink.  This completes the inductive step of the proof.
\end{proof}


\section{The Sato-Levine invariant and a lower bound on $u_k(L)$}\label{sect: lower bound}


In this section we produce some lower bounds on $u_k(L)$, the minimal number of crossing changes needed to transform \(L\) to a \(C_k\) trivial link.  We first produce explicit lower bounds on $u_3(L)$ and $u_4(L)$.  Then, as \(C_{k+1}\) triviality implies \( C_{k}\) triviality,  $u_k(L)$ is non-decreasing in $k$, so this will produce lower bounds on $u_k(L)$ for all $k\ge 4$.  
  
Theorem~\ref{thm:main lk=1} concludes that for any 2-component link $L$, $u_4(L)\le 4$.  We produce a 2-component link with vanishing linking number for which $u_4(L)\ge 3$.  
The link of Figure~\ref{fig: 2-component example} satisfies all of the requirements of the following result:


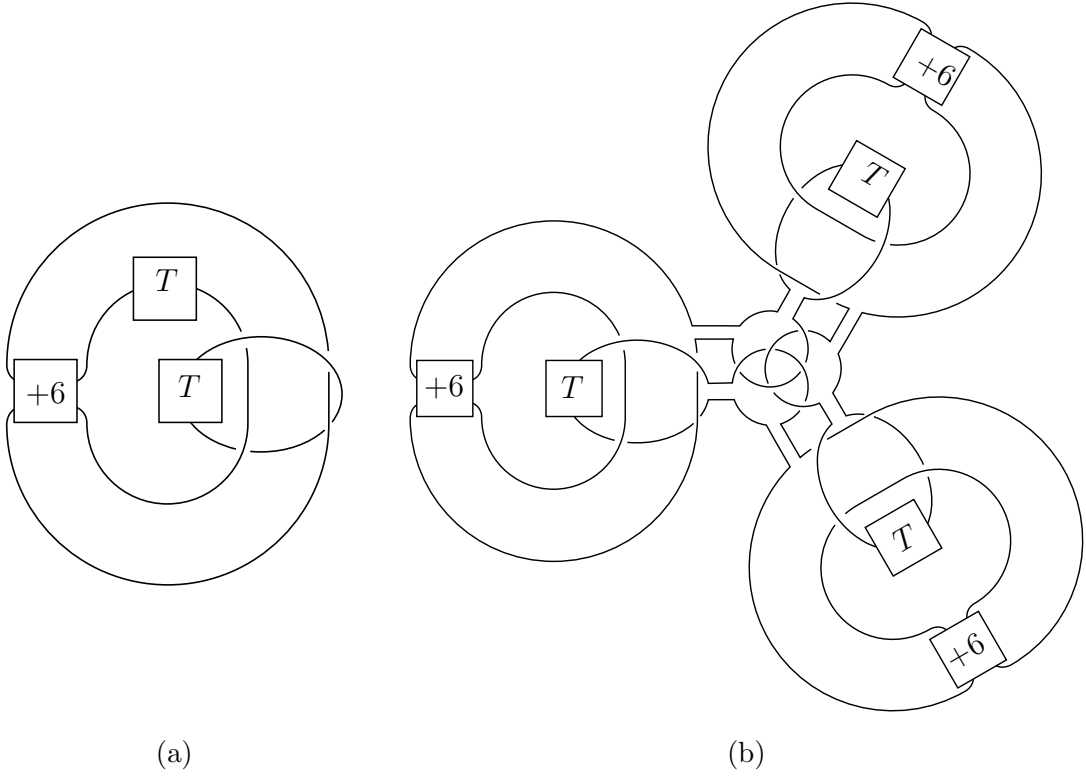
\begin{figure}
     \centering\begin{subfigure}[t]{0.3\textwidth}
         \centering
         \begin{tikzpicture}
         \node at (0,0){\includegraphics[width=.9\textwidth]{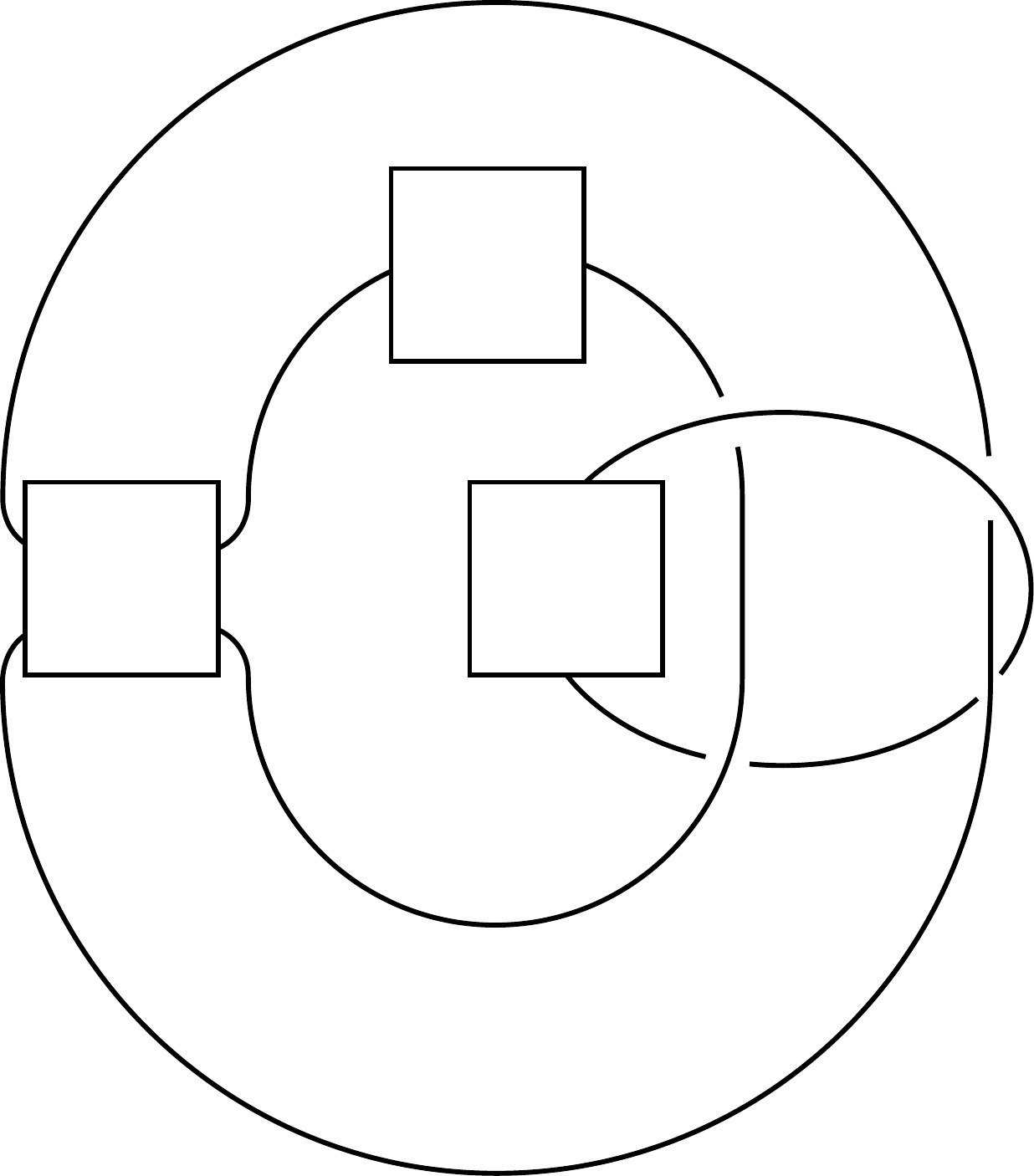}};
         \node at (-1.7,0) {$+6$};
         \node at (.2,0.1) {$T$};
         \node at (-.1,1.5) {$T$};
         \node at (0,-4.2) {};
         \end{tikzpicture}
         \caption{}\label{fig: 2-component example}
         \end{subfigure}
         \begin{subfigure}[t]{0.6\textwidth}
         \centering
         \begin{tikzpicture}
         \node at (0,0){\includegraphics[width=.9\textwidth]{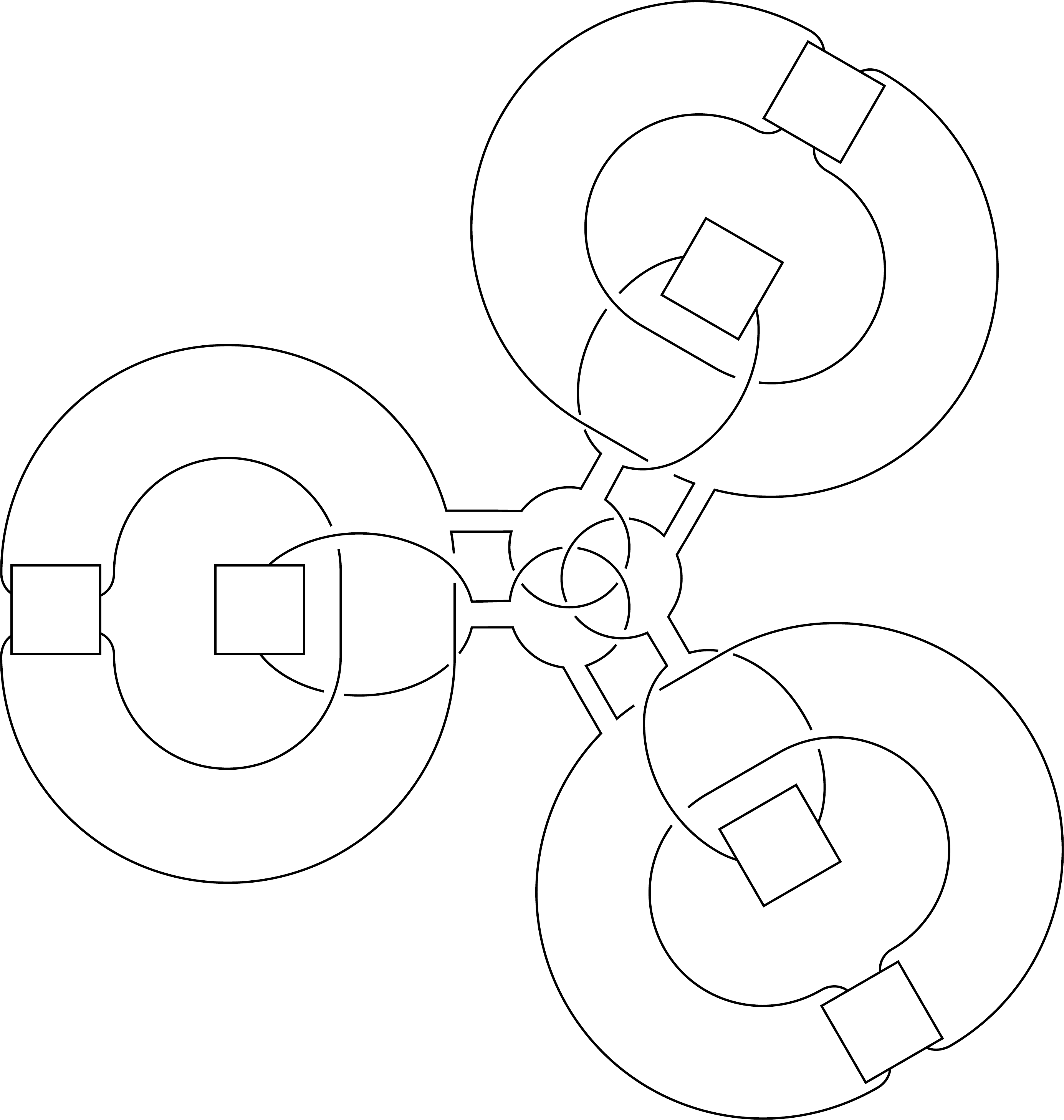}};
         
         \node at (-4,-.4) {$+6$};
         \node at (-2.3,-.4) {$T$};
         
         \node[rotate=-120+90] at (2.5,3.8) {$+6$};
         \node[rotate=-120+90] at (1.7,2.4) {$T$};
         
         \node[rotate=120-90] at (2.9,-3.9) {$+6$};
         \node[rotate=120-90] at (2.1,-2.4) {$T$};
         \end{tikzpicture}
         \caption{}\label{fig: 3-component example}
     \end{subfigure}
     
     \caption{Left: A 2-component link $L$ with $\mu_{1122}(L)=6$ and whose every component is a trefoil.  The ``+6" in the box indicates 6 full right handed twists and the $T$ indicates tying in a trefoil knot.  Right: a 3-component link whose every 2-component link is the link of (a) and with $\mu_{123}(L)=\pm1$}
     \end{figure}

\begin{theorem}\label{thm: c2>=3}
Let $L=L_1\cup L_2$ be a 2-component link with $\lk(L_1,L_2)=0$, where neither $L_1$ nor $L_2$ is $C_4$-trivial, $\mu_{1122}(L)>0$, and $\mu_{1122}(L)\equiv 6 \mod 8$.  Then $u_4(L)\ge 3$.  
\end{theorem}  
\begin{proof}

It suffices to show that no pair of crossing changes can reduce $L$ to a $C_4$-trivial link.  First note that by assumption, neither $L_1$ nor $L_2$ is $C_4$-trivial. Thus, if such a pair exists then it must consist of one self-crossing change on $L_1$ and one self-crossing change on $L_2$.  

  By \cite[Theorem 1.3]{JinThesis} (See also \cite[Theorem 4.1]{Li11}) each of these crossing changes alters $\mu_{1122}$ by $\pm k^2$ for some $k\in \Z$.  So if $L'$, the result of these two crossing changes, is $C_3$-trivial, then $\mu_{1122}(L')=0\) and so $\mu_{1122}(L)=\pm k^2\pm\ell^2$ for some $k,\ell\in \Z$.  Since $\mu_{1122}(L)>0$, we conclude that $\mu_{1122}(L)=+ k^2\pm\ell^2$

We now work over $\Z/8$, By direct inspection, the only perfect squares over $\Z/8$ are $0$, $1$, and $4$.  By further computation, $k^2\pm \ell^2$ never takes on the value of $6$ in $\Z/8$.  
\end{proof}

For 3-component links we add to our analysis the triple linking number, $\mu_{123}(L)$.  To see an example to which this theorem applies, see Figure~\ref{fig: 3-component example}.

\begin{theorem}\label{thm: c3>=6}
Let $L$ be a 3-component link with vanishing pairwise linking numbers, whose every component is not $C_3$-trivial, and which has $\mu_{123}(L)\neq 0$.  Then $u_3(L)\ge 5$.  If additionally, $\mu_{1122}(L_i\cup L_j)=6$ for all $i,j$  then $u_4(L)\ge 6$.  
\end{theorem}
\begin{proof}
Consider any sequence of crossing changes transforming $L$ to a $C_3$-trivial link.  Since each component is not $C_3$-trivial, this sequence must include a self-crossing change on $L_1$, a self-crossing change on $L_2$ and a self-crossing change on $L_3$.    Since $\mu_{123}(L)\neq 0$, and $\mu_{123}$ is preserved by self-crossing changes \cite{M1}, this sequence must also contain a crossing change between two components say $L_1$ and $L_2$.  The first crossing change between components $L_1$ and $L_2$ will result in $\lk(L_1,L_2)= \pm1$ so that at least one more crossing change between $L_1$ and $L_2$ is needed to get a $C_3$-trivial link.  

Now add the assumption that $\mu_{1122}(L_i\cup L_j)=6$ for all $i,j$.  So far we have seen that any sequence of crossing changes transforming $L$ to a $C_4$-trivial link must contain at least one self-crossing change on each component and two crossing changes between different components, say $L_1$ and $L_2$.  We have changed $L_1\cup L_3$ and $L_2\cup L_3$ each by only two self-crossing changes, which as we observed in Theorem~\ref{thm: c2>=3} means that these two sublinks cannot be $C_4$-trivial.  Thus, at least one more crossing change is needed.  
\end{proof}

We now give examples that prove that there are $n$-components links for which $u_4(L)$ grows quadratically with $n$. 

\begin{theorem}\label{thm cn lower bound} For any $n\ge 4$ there is an $n$-component link $L$ with $u_4(L)\ge 2\ceil{\frac{n(n-2)}{3}}+n$.  
\end{theorem}
\begin{proof}
In \cite[Proposition 7.4]{BDMOV2025} there appears a 4-component link $J$ which cannot be reduced to  homotopy trivial link in fewer than $6$ crossing changes and all of those crossing changes must be between different components.  As a $C_4$-trivial $4$-component link is homotopy trivial, $u_4(J)\ge 6$.  

In the proof of \cite[Theorem 7.2]{BDMOV2025} it is shown that if $L$ is a link whose every 4-component sublink is $J$, then any sequence of crossing changes transforming $L$ to a link whose every 4-component sublink is homotopy trivial requires at least $2\ceil{\frac{n(n-2)}{3}}$ crossing changes between different components.    If we also require each component of $L$ to not be $C_4$-trivial, then any sequence of crossing changes resulting in a $C_4$-trivial link will need $n$ additional self-crossing changes, completing the proof. 
\end{proof}

\bibliographystyle{plain}
\bibliography{biblio}

\end{document}